\documentclass[runningheads]{llncs}      

\usepackage[T1]{fontenc}
\usepackage{amssymb,amsmath,amsfonts,mathrsfs}
\usepackage[utf8]{inputenc}
\usepackage[english]{babel}
\usepackage{graphicx}
\usepackage{color}
\usepackage{doi}
\usepackage{commath}

\DeclareMathOperator{\RR}{\mathbb{R}}

\DeclareMathOperator{\ZZ}{\mathbb{Z}}

\DeclareMathOperator{\BC}{\mathcal{B}}

\DeclareMathOperator{\TC}{\mathcal{T}}

\DeclareMathOperator{\PC}{\mathcal{P}}
\DeclareMathOperator{\DC}{\mathcal{D}}
\DeclareMathOperator{\SC}{\mathcal{S}}

\DeclareMathOperator{\UC}{\mathcal{U}}

\DeclareMathOperator{\JC}{\mathcal{J}}
\DeclareMathOperator{\IC}{\mathcal{I}}

\DeclareMathOperator{\BS}{\mathscr{B}}

\DeclareMathOperator{\DS}{\mathscr{D}}

\DeclareMathOperator{\RS}{\mathscr{R}}

\DeclareMathOperator{\LS}{\mathscr{L}}

\DeclareMathOperator{\SScr}{\mathscr{S}}

\DeclareMathOperator{\rank}{rank}

\DeclareMathOperator{\width}{width}

\DeclareMathOperator{\vertex}{vert}
\DeclareMathOperator{\lcm}{lcm}

\DeclareMathOperator{\poly}{poly}

\DeclareMathOperator{\paral}{Par}

\DeclareMathOperator{\BZero}{\mathbf 0}

\newcommand*{\intint}[2][1]{\{#1, \dots, #2\}}

\begin{document}

\title{Enumeration and Unimodular Equivalence of Empty Delta-Modular Simplices \thanks{The article was prepared under financial support of Russian Science Foundation grant No 21-11-00194.}
}

%
%
\author{D.~V.~Gribanov\inst{1,2}\orcidID{0000-0002-4005-9483}}
%
%
\institute{
Lobachevsky State University of Nizhny Novgorod, 23 Gagarina Avenue, Nizhny Novgorod, 603950, Russian Federation \and
National Research University Higher School of Economics, 25/12 Bolshaja Pecherskaja Ulitsa, Nizhny Novgorod, 603155, Russian Federation \email{dimitry.gribanov@gmail.com}
}
\maketitle              
\begin{abstract}

Consider a class of simplices defined by systems $A x \leq b$ of linear inequalities with \emph{$\Delta$-modular} matrices. A matrix is called \emph{$\Delta$-modular}, if all its rank-order sub-determinants are bounded by $\Delta$ in an absolute value. In our work we call a simplex \emph{$\Delta$-modular}, if it can be defined by a system $A x \leq b$ with a $\Delta$-modular matrix $A$. And we call a simplex \emph{empty}, if it contains no points with integer coordinates. In literature, a simplex is called \emph{lattice-simplex}, if all its vertices have integer coordinates. And a lattice-simplex called \emph{empty}, if it contains no points with integer coordinates excluding its vertices. 

Recently, assuming that $\Delta$ is fixed, it was shown in \cite{OnCanonicalProblems_Grib} that the number of $\Delta$-modular empty simplices modulo the unimodular equivalence relation is bounded by a polynomial on dimension. We show that the analogous fact holds for the class of $\Delta$-modular empty lattice-simplices. As the main result, assuming again that the value of the parameter $\Delta$ is fixed, we show that all unimodular equivalence classes of simplices of the both types can be enumerated by a polynomial-time algorithm. As the secondary result, we show the existence of a polynomial-time algorithm for the problem to check the unimodular equivalence relation for a given pair of $\Delta$-modular, not necessarily empty, simplices.

\end{abstract}

\section{Introduction}

On the one hand, simplices are quite simple objects, since they form a class of solid polytopes with the minimum possible number of vertices. At the same time, the simplices are substantial and fundamental geometric and combinatorial objects. For example, the unbounded knapsack problem can be modeled as the integer linear optimization problem on a simplex; Simplices are fundamental building blocks of any triangulation of a given polytope; Simplices can be used as universal and quite simple bounding regions; And etc.

In our work we consider only the simplices $\SC$ defined by $\SC = \{ x \in \RR^n \colon A x \leq b \}$, where $A \in \ZZ^{(n+1)\times n}$, $\rank(A) = n$ and $b \in \ZZ^{n+1}$. Our main interest consists of various problems associated with the set $\SC \cap \ZZ^n$. For example, in the \emph{integer feasibility problem} we need to decide, ether $\SC \cap \ZZ^n \not= \emptyset$ or not. This problem is naturally $N\!P$-complete, because the classical $N\!P$-complete unbounded subset-sum problem can be represented as an integer feasibility problem in a simplex. In the more general $N\!P$-hard \emph{integer linear optimization problem} we need to find a point $x^* \in \SC \cap \ZZ^n$ such that $c^\top x^* = \max\{c^\top x \colon x \in \SC \cap \ZZ^n\}$, where $c \in \ZZ^n$ is an arbitrary vector, or decide that $\SC \cap \ZZ^n = \emptyset$.  However, assuming that dimension is fixed, due to seminal work of Lenstra \cite{Lenstra} (for more up-to-date and universal algorithms see, for example, \cite{DadushFDim,FixedDimILPSurv_Eisen,DadushDis,Convic,ConvicComp,DConvic,FixedDimILP_Eisen}), the integer feasibility and integer linear optimization problems, for any polytope defined by a linear inequalities system, can be solved by a polynomial-time algorithm. In the \emph{integer points counting problem} we need to calculate the value of $|\SC \cap \ZZ^n|$, which is, by the same reasons, a $\#\!P$-hard problem. Due to seminal work of Barvinok \cite{Barv_Original} (see also the works \cite{BarvPom,BarvWoods,OnBarvinoksAlg_Dyer,HalfOpen} and the books \cite{BarvBook,AlgebracILP,counting_Lasserre_book,continuous_discretely}), assuming that dimension is fixed, the last problem can be solved by a polynomial-time algorithm for general polytopes.

Important classes of simplices (and even more general polytopes) are \emph{lattice-simplices} and \emph{empty lattice-simplices}. A simplex $\SC$ (or a general polytope) is called \emph{lattice-simplex}, if $\vertex(\SC) \subseteq \ZZ^n$. A lattice-simplex $\SC$ (or a general lattice-polytope) is called \emph{empty}, if $\SC \cap \ZZ^n = \vertex(\SC)$. In our work, a simplex $\SC$, which is not a lattice-simplex, is called \emph{empty}, if $\SC \cap \ZZ^n = \emptyset$. 

The empty lattice-polytopes appear in many important works concerning the theory of integer programming. They are one of the central objects in theory of cutting planes \cite{IntersectionCuts_Init,IneqFromTwoRows_SimplexTable,MinimalValidIneq,MixedRelaxations_LatticeFree,GeometricApproach_Cutting}, and they can be used to construct optimality certificates in convex integer optimization \cite{BlairConstructive_1,BlairConstructive_2,BHull,Wolsey_Duality,Conic_Dual,OptimalitySertif_HellyNumbers,DualityForMixedConvex,LatticeFree_Gradient}. From the perspective of algebraic geometry, empty simplices are almost in bijection with the terminal quotient singularities, see, for example, \cite{Empty4Simplex_Init,TerminalSingul34,Empty4Simplex_Complete,Borisov_Quotientsingularities}. The problem to classify all $4$-dimensional empty lattice-simplices was completely solved in the paper \cite{Empty4Simplex_Complete}. Earlier results in this direction could be found, for example, in the works \cite{Empty4Simplex_Init,IntroEmptyLatticeSimplicies,Ziegler_width,Empty4Simplex_Width2}. The $3$-dimensional classification problem was solved by White \cite{LatticeTetrahedra} (see also \cite{IntroEmptyLatticeSimplicies,TerminalSingul34}). Important structural properties of general $3$-dimensional empty lattice-polytopes are presented in \cite{StructureOfSimpleSetsZ3}.

An important parameter that is strongly connected with the emptiness property of a simplex, lattice-simplex, or a general lattice-free convex body is the \emph{lattice-width} or, simply, the \emph{width}:
\begin{definition}\label{width_def}
    For a convex body $\PC \subseteq \RR^n$ and a vector $c \in \ZZ^n \setminus \{\BZero\}$, define
    \begin{gather*}
    \width_c(\PC) = \max\limits_{x \in \PC} c^\top x - \min\limits_{x \in \PC} c^\top x, \quad\text{and}\\
    \width(\PC) = \min\limits_{c \in \ZZ^n \setminus \{\BZero\}} \bigl\{\width_c(\PC)\bigr\}.
    \end{gather*}

    The value of $\width(\PC)$ is called the \emph{lattice-width} (or simply the \emph{width}) of $\PC$.
\end{definition}

This connection can be derived from the famous work of A.~Khinchin, and it is known by the name \emph{Flatness Theorem}.
\begin{theorem}[Flatness Theorem, A.~Khinchin \cite{flatness_Khinchine}]
    Let $\PC \subseteq \RR^n$ be a convex body with $\PC \cap \ZZ^n = \emptyset$.
    $$
    \text{Then,}\quad \width(\PC) \leq \omega(n),
    $$ where $\omega(n)$ is some function that depends only on $n$.
\end{theorem}
Many proofs of the Flatness Theorem have been presented in \cite{Babai_flatness,Lagarias_flatness,CoveringMinima_flatness,Banaszczyk_transferene,EllFlatness,ViaLocalTheoryBanach}, each giving different asymptotic estimates on $\omega(n)$. Combining the works of Banaszczyk, Litvak, Pajor \& Szarek \cite{ViaLocalTheoryBanach} and Rudelson \cite{Rudelson_flatness}, the currently best upper bound is $\omega(n) = n^{4/3} \cdot \log^{O(1)}(n)$. With restriction to simplices, it was shown \cite{ViaLocalTheoryBanach} that $\omega(n) = O(n \log(n))$. The same bound on $\omega(n)$ with restriction to symmetric convex bodies, was shown by Banaszczyk \cite{EllFlatness}. There is a sequence of works \cite{Kantor_width,IntroEmptyLatticeSimplicies,HollowPoly_LargeWidth,EmptySimplices_widthMore_d,EmptySimplices_width2d} devoted to establish a lower bound for $\omega(n)$. The state of the art result of this kind belongs to the work \cite{EmptySimplices_width2d} of Mayrhofer, Schade \& Weltge, where the construction of an empty lattice-simplex of the width $2n - o(d)$ is presented. It is conjectured in \cite{ViaLocalTheoryBanach} that $\omega(n) = \Theta(n)$, even restricting the values of $\omega(n)$ to simplices. To the best of our knowledge, the only one non-trivial case, when the conjecture was verified, is the case of ellipsoids, due to \cite{EllFlatness}.

Due to \cite{IntroEmptyLatticeSimplicies}, it is a $N\!P$-complete problem to decide that $\width(\SC) \leq 1$ for a given lattice-simplex $\SC$. However, as it was noted in \cite{IntroEmptyLatticeSimplicies}, assuming that dimension is fixed, the problem again can be solved by a polynomial-time algorithm.

\subsection{The class of $\Delta$-modular simplices}

It turns out that it is possible to define an interesting and substantial parameter of the system $A x \leq b$, which significantly affects complexity of the considered problems. 

\begin{definition}\label{delta_def}
For a matrix $A \in \ZZ^{m \times n}$, by $$
\Delta_k(A) = \max\left\{\abs{\det (A_{\IC \JC})} \colon \IC \subseteq \intint m,\; \JC \subseteq \intint n,\; \abs{\IC} = \abs{\JC} = k\right\},
$$ we denote the maximum absolute value of determinants of all the $k \times k$ sub-matrices of $A$. Here, the symbol $A_{\IC \JC}$ denotes the sub-matrix of $A$, which is generated by all the rows with indices in $\IC$ and all the columns with indices in $\JC$.

Note that $\Delta_1(A) = \|A\|_{\max} = \max_{i j} |A_{i j}|$. Additionally, let $\Delta(A) = \Delta_{\rank(A)}(A)$. The matrix $A$ with $\Delta(A) \leq \Delta$, for some $\Delta > 0$, is called \emph{$\Delta$-modular}. The polytope $\PC$, which is defined by a system $A x \leq b$ with a $\Delta$-modular matrix $A$, is called \emph{$\Delta$-modular polytope}.     
\end{definition}

Surprisingly, all the mentioned problems restricted on $\Delta$-modular simplices with $\Delta = \poly(\phi)$, where $\phi$ is the input length, can be solved by a polynomial time algorithms. Such algorithms are also known as the FPT-algorithms. Definitely, restricting the problems on simplices, due to \cite[Theorem~14]{OnCanonicalProblems_Grib} the integer feasibility problem can be solved with $O\bigl(n + \min\{n,\Delta\} \cdot \Delta \cdot \log(\Delta) \bigr)$ operations. Due to \cite[Corollary~9]{OnCanonicalProblems_Grib} (see also \cite{FPT_Grib}), the integer linear optimization problem can be solved with $O\bigl(n + \Delta^2 \cdot \log(\Delta) \cdot \log(\Delta_{\gcd})\bigr)$ operations. Due to Gribanov \& Malyshev \cite{Counting_FPT_Delta} and a modification from Gribanov, Malyshev \& Zolotykh \cite{SparseILP_Gribanov}, the integer points counting problem can be solved with $O(n^4 \cdot \Delta^3)$ operations. Due to Gribanov, Malyshev, Pardalos \& Veselov \cite{FPT_Grib} (see also \cite{WidthSimplex_Grib}), the lattice-width can be computed with $\poly(n,\Delta,\Delta_{ext})$-operations, where $\Delta_{ext} = \Delta(A\,b)$ and $(A\,b)$ is the system's extended matrix. Additionally, for empty simplices, the complexity dependence on $\Delta_{ext}$ can be avoided, which gives the bound $\poly(n,\Delta)$. The analogous result for lattice-simplices defined by convex hulls of their vertices is presented in Gribanov, Malyshev \& Veselov \cite{WidthConv_Grib}.

Additionally, it was shown by Gribanov \& Veselov \cite{Width_Grib} that it is possible to prove a variant of the Flatness Theorem for simplices such that the function $\omega$ will depend only on the sub-determinants spectrum of $A$ instead of $n$. More precisely, the following statement is true: for any empty simplex $\SC$, the inequality $\width(\SC) < \Delta_{\min}(A)$ holds, where $\Delta_{\min}(A)$ is the minimum between all rank-order nonzero sub-determinants of $A$ taken by an absolute value. Additionally, if the inequality is not satisfied, then some integer point inside $\SC$ can be found by a polynomial-time algorithm. The last inequality for empty simplices was improved to $\width(\SC) < \lfloor \Delta_{\min}/2 \rfloor$ in the work \cite{DeltaSimplexWidth_Improved} of Henk, Kuhlmann \& Weismantel. For a general $\Delta$-modular empty polytope $\PC$, the inequality $\width(\SC) \leq (\Delta_{\lcm}(A)-1) \cdot \frac{\Delta(A)}{\Delta_{\gcd}(A)} \cdot (n+1)$ was shown in \cite{Width_Grib}, where $\Delta_{\gcd}(A)$ and $\Delta_{\lcm}(A)$ are the gcd and lcm functions of rank-order sub-determinants of $A$ taken by an absolute value. Additionally, if the inequality is not satisfied, then $\PC$ contains a lattice-simplex, and some vertex of this simplex can be found by a polynomial-time algorithm. Recently, in the work \cite{WidthDelta_Improved} of Celaya, Kuhlmann, Paat \& Weismantel, the last inequality for empty polytopes was improved to $\width(\SC) < \frac{4n+2}{9} \cdot \Delta(A)$.

\subsection{The main motivation of the work and results}

To formulate the main problems and results, we need first to make some additional definitions. For a unimodular matrix $U \in \ZZ^{n \times n}$ and an integer vector $x_0 \in \ZZ^n$, the affine map $\UC(x) = U x + x_0$ is called \emph{unimodular}. Simplices $\SC_1$ and $\SC_2$ are called \emph{unimodular equivalent}, if there exists an unimodular affine map $\UC$ such that $\UC(\SC_1) = \SC_2$. It is easy to see that this relation partitions all the simplices into equivalence classes.

It was shown in \cite[Theorem~16]{OnCanonicalProblems_Grib} that the number of empty $\Delta$-modular simplices modulo the unimodular equivalence relation is bounded by
\begin{equation}\label{old_bound}
    \binom{n+\Delta-1}{\Delta-1} \cdot \Delta^{\log_2(\Delta) + 2},
\end{equation} which is a polynomial, for fixed $\Delta$. This result motivates the following
\begin{problem}[Empty $\Delta$-modular Simplices Enumeration]\label{enum_problem}
    For a given value of the parameter $\Delta$, enumerate all the unimodular equivalence classes of $\Delta$-modular empty simplices and empty lattice-simplices.
\end{problem}

To solve the last problem, it is important to have an algorithm that can check the unimodular equivalence relation.
\begin{problem}[Unimodular Equivalence Checking]\label{equiv_problem}
For given simplices $\SC$ and $\TC$, determine whether the simplices $\SC$ and $\TC$ are unimodular equivalent. 
\end{problem}

As the main results we show that the both problems can be solved by polynomial-time algorithms, assuming that the value of $\Delta$ is fixed. Unfortunately, the new algorithms do not belong to the FPT-class. The formal definitions of our results are presented in the following theorems.  

\begin{theorem}\label{equiv_th}
Let $\SC$ and $\TC$ be $\Delta$-modular simplices. The unimodular equivalence relation for $\SC$ and $\TC$ can be checked by an algorithm with the complexity 
$$
n^{\log_2(\Delta)} \cdot \Delta! \cdot \poly(n,\Delta),
$$ which is polynomial, for fixed values of $\Delta$.
\end{theorem}

\begin{theorem}\label{enum_th}
    We can enumerate all representatives of the unimodular equivalence classes of the $\Delta$-modular empty simplices and empty lattice-simplices by an algorithm with the complexity bound  
    $$
    O\bigl(n + \Delta\bigr)^{\Delta-1} \cdot n^{\log_2(\Delta)+O(1)},
    $$
    which is polynomial, for fixed values of $\Delta$.
\end{theorem}
The proof is given in the next Section \ref{proofs_sec}.

\section{Proof of Theorems \ref{equiv_th} and \ref{enum_th}}\label{proofs_sec}

\subsection{The Hermite Normal Form}\label{HNF_subs}

Let $A \in \ZZ^{m \times n}$ be an integer matrix of rank $n$, assuming that the first $n$ rows of $A$ are linearly independent. It is a known fact (see, for example, \cite{Schrijver,HNFOptAlg}) that there exists a unimodular matrix $Q \in \ZZ^{n \times n}$, such that $A = \binom{H}{B} Q$, where $B \in \ZZ^{(m-n) \times n}$ and $H \in \ZZ_{\geq 0}^{n \times n}$ is a lower-triangular matrix, such that $0 \leq H_{i j} < H_{i i}$, for any $i \in \intint n$ and $j \in \intint{i-1}$. The matrix $\binom{H}{B}$ is called the \emph{Hermite Normal Form} (or, shortly, the HNF) of the matrix $A$. Additionally, it was shown in \cite{FPT_Grib} that $\|B\|_{\max} \leq \Delta(A)$ and, consequently, $\bigl\|\binom{H}{B}\bigr\|_{\max} \leq \Delta(A)$. Near-optimal polynomial-time algorithm to construct the HNF of $A$ is given in the work \cite{HNFOptAlg} of Storjohann \& Labahn.


\subsection{Systems in the normalized form and their enumeration}\label{norm_subs}

\begin{definition}\label{norm_def}
Assume that a simplex $\SC$ is defined by a system $A x \leq b$, for $A \in \ZZ^{(n+1)\times n}$ and $b \in \ZZ^{n+1}$.

The system $A x \leq b$ is called \emph{normalized}, if it has the following form:
$$
\binom{H}{c^\top} x \leq \binom{h}{c_0}, \quad\text{where}
$$
\begin{enumerate}
    \item The matrix $A = \binom{H}{c^\top}$ is the HNF of some integer $(n+1)\times n$ matrix and $|\det(H)| = \Delta(A)$;
    
    \item The matrix $H$ has the form $H = \begin{pmatrix} 
    I_{s} & \BZero_{s \times k} \\
    B & T \\
    \end{pmatrix}$, where $k+s = n$, $k \leq \log_2(\Delta)$, the columns of $B$ are lexicographically sorted, $T$ has a lower triangular form and $T_{i i} \geq 2$, for any $i \in \intint k$;

    \item For $i \in \intint n$, $0 \leq h_i < H_{i i}$. Consequently, $h_i = 0$, for $i \in \intint s$;

    \item For any inequality $a^\top x \leq a_0$ of the system $A x \leq b$, $\gcd(a, a_0) = 1$;
    
    \item $c \in \paral(-H^\top)$, where $\paral(M) = \bigl\{ M t \colon t \in (0,1]^n \bigr\}$, for arbitrary $M \in \ZZ^{n \times n}$;

    \item $\|A\|_{\max} \leq \Delta$.
    
\end{enumerate}
\end{definition}

\begin{lemma}\label{norm_prop_lm}
    The following properties of normalized systems hold:
    \begin{enumerate}
        \item The $5$-th and $6$-th conditions of Definition \ref{norm_def} are redundant, i.e. they follow from the remaining conditions and properties of $\SC$;

        \item Let $\SC$ be the simplex used in Definition \ref{norm_def}. Assume that $\SC$ is empty simplex or empty lattice-simplex, then $|c_0 - c^\top v| \leq \Delta$, where $v = H^{-1} h$ is the opposite vertex of $\SC$ with respect to the facet induced by the inequality $c^\top x \leq c_0$. 
    \end{enumerate}
\end{lemma}
\begin{proof}
    Let us show that the condition $c \in \paral(-H^\top)$ is redundant. Definitely, since $\SC$ is a simplex, it follows that $c = -H^\top t$, for some $t \in \RR^n_{>0}$. Assume that there exists $i \in \intint n$ such that $t_i > 1$. Let $M$ be the matrix formed by rows of $H$ with indices in the set $\intint{n} \setminus \{i\}$ and a row $c^\top$. Then, $|\det(M)| = t_i \cdot |\det(H)| > \Delta$, which contradicts to the $\Delta$-modularity property of $\SC$. Hence, $0 < t_i \leq 1$, for all $i \in \intint n$, and $c \in \paral(-H^\top)$.

    As it was mentioned in Subsection \ref{HNF_subs}, for each $\Delta$-modular matrix $A$ that is reduced to the HNF, the inequality $\|A\|_{\max} \leq \Delta$ holds. So the $6$-th condition just follows from the first condition and the $\Delta$-modularity property of $A$.

    Let us prove the second property claimed in the lemma. In the case when $\SC$ is empty, the inequality $|c_0 - c^\top v| < \Delta$ was proven in \cite[Theorem~7]{Width_Grib} (see also \cite[Corollary~10]{OnCanonicalProblems_Grib}). Let us prove the inequality $|c^\top v - c_0| \leq \Delta$ in the case when $\SC$ is empty lattice-simplex. Consider a point $p = v - u$, where $u$ is the first column of the adjugate matrix $H^* = \det(H) \cdot H^{-1}$. Note that $p$ is an integer feasible solution for the sub-system $H x \leq h$. Since $\|c^\top H^*\|_{\infty} \leq \Delta$, it holds that $|c^\top p - c^\top v| \leq \Delta$. Then, the inequality $|c^\top v - c_0| > \Delta$ implies that $p \in \SC \setminus \vertex(\SC)$, which is the contradiction.
\end{proof}

The following lemma is one of the elementary building blocks for an algorithm to construct normalized systems.
\begin{lemma}\label{rhs_lm}
    Let $H \in \ZZ^{n \times n}$ be a non-degenerate matrix reduced to the HNF and $b \in \ZZ^n$. Then, there exists a polynomial-time algorithm that computes the unique translation $x \to x + x_0$ that maps the system $H x \leq b$ to the system $H x \leq h$ with the property $0 \leq h_i < H_{i i}$, for any $i \in \intint n$. 
\end{lemma}
\begin{proof}
Denote the $i$-th standard basis vector of $\RR^n$ by $e_i$. Consider the following algorithm: set $b^{(0)} := b$ and, for the index $i$ increasing from $1$ to $n$, we sequentially apply the integer translation $x \to x + \lfloor b^{(i-1)}/H_{i i} \rfloor \cdot e_i$ to the system $A x \leq b$, where $b^{(i)}$ denotes the r.h.s. of $A x \leq b$ after the $i$-th iteration. We have $b^{(i+1)} := b^{(i)} - \lfloor b^{(i)}/H_{i i} \rfloor H_{i}$, where $H_{i}$ is the $i$-th column of $H$. If we denote the resulting system r.h.s. by $h$, then $h = b^{(n)}$ and $h_i = b_i \bmod H_{i i}$, for $i \in \intint n$. Clearly, the algorithm that returns $h$ and the cumulative integer translation vector $x_0$ is polynomial-time. 

Note that the vectors $b$ and $h$ are additionally connected by the formula $h = b - H x_0$. Due to the HNF triangle structure, and since $h_1 = b_1 \bmod H_{1 1}$, it follows that $(x_0)_1 = \lfloor b_0/H_{1 1} \rfloor$, so the first component of $x_0$ is unique. Now, using the induction principle, assume that the first $k$ components of $x_0$ are uniquely determined by $H$ and $b$, and let us show that the same fact holds for $(x_0)_{k+1}$. Clearly, 
$$
h = b - H_{\intint k} (x_0)_{\intint k} - H_{\intint[k+1]{n}} (x_0)_{\intint[k+1]{n}}.
$$ Denoting $\hat h = h + H_{\intint k} (x_0)_{\intint k}$, we have
\begin{equation*}
\hat h = b - H_{\intint[k+1]{n}} (x_0)_{\intint[k+1]{n}}.
\end{equation*}
Note that $\hat h_{k+1} = h_{k+1} = b_{k+1} \bmod H_{(k+1)(k+1)}$. Again, due to the HNF triangle structure, it holds that $(x_0)_{k+1} = \lfloor b_{k+1} / H_{(k+1)(k+1)} \rfloor$, which finishes the proof.
\end{proof}

The following lemma proposes an algorithm for normalized systems construction. Additionally, it proves the fact that all equivalent normalized systems can be constructed by this way.
\begin{lemma}[The Normalization Algorithm]\label{form_lm}
Let $\SC$ be a $\Delta$-modular simplex defined by a system $A x \leq b$.

The following propositions hold:
\begin{enumerate}

    \item For a given base $\BC$ of $A$ with $|\det(A_{\BC})| = \Delta$, there exists a polynomial-time algorithm that returns a unimodular equivalent normalized system $
\binom{H}{c^\top} x \leq \binom{h}{c_0}
$ and a permutation matrix $P \in \ZZ^{n \times n}$ such that $H$ is the HNF of $P A_{\BC}$. Let this algorithm be named as the \emph{Normalization Algorithm};

\item Any normalized system $
\binom{H}{c^\top} x \leq \binom{h}{c_0}
$ that is unimodular equivalent to $\SC$ can be obtained as an output of the normalization algorithm with the input base $\intint n$ of a permuted system $P A x \leq P b$, for some permutation matrix $P \in \ZZ^{(n+1)\times(n+1)}$.
\end{enumerate}
\end{lemma}
\begin{proof}
Let us describe the normalization algorithm. First of all, we represent $A_{\BC} = H Q$, where $H \in \ZZ^{n \times n}$ is the HNF of $A_{\BC}$ and $Q \in \ZZ^{n \times n}$ is a unimodular matrix. After some straightforward additional permutations of rows and columns, we can assume that the original system is equivalent to the system
    $$
    \binom{H}{c^\top} x \leq \binom{h}{c_0},\quad\text{where $H = \begin{pmatrix} 
    I_{s} & \BZero_{s \times k} \\
    B & T \\
    \end{pmatrix}$},
    $$ $s + k = n$, $k \leq \log_2(n)$ and $T$ has a lower-triangular form with $T_{i i} \geq 2$, for $i \in \intint k$.
    
    To fulfil all the conditions of Definition \ref{norm_def}, we need some additional work with the matrix $B$ and vector $h$. We can fulfill the second condition in Definition \ref{norm_def} for $B$ just using any polynomial-time sorting algorithm. The new system will be unimodular equivalent to the original one, because column-permutations of the first $s$ columns can be undone by row-permutations of the first $s$ rows. Next, the third condition for $h$ can be satisfied just using the algorithm of Lemma \ref{rhs_lm}. Finally, the $4$-th condition can be easily satisfied just by dividing each line of the system on the corresponding value of the gcd function. Since all the steps can be done using polynomial-time algorithms, the whole algorithm is polynomial-time, which proves the first proposition of the lemma.

    Let us prove the second proposition. Assume that some unimodular equivalent normalized system $A' x \leq b'$ is given. 
Let $\SC'$ be the simplex defined by $A' x \leq b'$, and let $\UC(x) = U x' + x_0$ be the unique unimodular affine map that maps $\SC$ into $\SC'$. It follows that there exists a bijection between the facets of $\SC$ and $\SC'$ induced by $\UC$. Since any inequality of the system $A x \leq b$ defines some facet of $\SC$, it directly follows that there exists a permutation matrix $P \in \ZZ^{(n+1) \times (n+1)}$ such that the systems $P A x \leq P b$ and $A' x \leq b'$ are equivalent with one to one correspondence of each inequality of both systems. Since two arbitrary inequalities are equivalent modulo some positive homogeneity multiplier, and due to the condition 4 of Definition \ref{norm_def}, it follows that there exists a diagonal matrix $D \in \ZZ^{(n+1)\times (n+1)}$ with strictly positive integer diagonal elements such that
\begin{equation}\label{D_eq}
 P A U = D A' \quad\text{and}\quad P (b - A x_0) = D b'.
\end{equation}
Let us use the normalization algorithm with the input system $P A x \leq P b$ and base $\BC = \intint n$. The algorithm consists of three parts: the first part constructs the resulting matrix, the second part constructs the resulting r.h.s., the third part removes homogeneity multipliers from rows of the resulting system. Let us consider how the input system $P A x \leq P b$ will look after the first part of the algorithm. Since $D A'$ is the HNF of $P A$, the system will be $D A' x \leq P b$. Now, we apply the second part of the algorithm using Lemma \ref{rhs_lm} to the system $D A' x \leq P b$. Due to the equalities \eqref{D_eq},  
$$
D b' = P b - P A x_0 = P b - P A U t = P b - D A' t,\quad \text{for $t = U^{-1} x_0 \in \ZZ^n$}.
$$
Due to Lemma \ref{rhs_lm}, it holds that $t$ is the unique integer translation vector that transforms the system $D A' x \leq P b$ to the system $D A' x \leq D b'$. Consequently, the second part of the algorithm will produce the system $D A' x \leq D b'$, which will be transformed to the resulting normalized system $A' x \leq b'$ after the third part of the algorithm. The proof is finished.


\end{proof}

The following lemma helps to enumerate all equivalent normalized systems. 
\begin{lemma}\label{enum_lm}
    For a given parameter $\Delta$, there exists an algorithm that constructs two families of simplices $\SScr$ and $\LS$ with the following properties:
    \begin{enumerate}
        \item Each simplex of the families $\SScr$ and $\LS$ is $\Delta$-modular, and is defined by a normalized system;
        
        \item The simplices of $\SScr$ are empty, and all the unimodular equivalence classes of $\Delta$-modular empty simplices are represented as simplices from $\SScr$;
        
        \item The simplices of $\LS$ are lattice and empty, and all the unimodular equivalence classes of $\Delta$-modular empty lattice-simplices are represented as simplices from $\LS$.
    \end{enumerate}

    The computational complexity of the algorithm is 
    \begin{multline*}
    \binom{n+\Delta-1}{\Delta-1} \cdot \Delta^{\log_2(\Delta)} \cdot \poly(n,\Delta) = \\
    = O\left(\frac{n + \Delta}{\Delta}\right)^{\Delta-1} \cdot n^{O(1)},    
    \end{multline*}
    which is polynomial, for any fixed $\Delta$.
\end{lemma}
\begin{proof}
    Due to Lemma \ref{form_lm}, to construct simplices of the sets $\SScr$ and $\LS$, it is sufficient to enumerate all the normalized systems $$
    \binom{H}{c^\top} x \leq \binom{h}{c_0},\quad\text{where $H = \begin{pmatrix}
        I_{s} & \BZero_{s \times k} \\
        B & T
    \end{pmatrix}$.}
    $$

    Let $\DC$ be the set of all nontrivial divisors of $\Delta$. Clearly, $\DC$ can be constructed in $\poly(\Delta)$-time. To enumerate all the normalized systems, the following simple scheme can be used:
    \begin{enumerate}
        \item Enumerate all the possible tuples $(d_1, d_2, \dots, d_k)$ of $\DC$ such that $d_1\cdot \ldots \cdot d_k = \Delta$. Tuples allow duplicates;
        \item For any fixed tuple $(d_1, d_2, \dots, d_k)$, enumerate the matrices $H = \begin{pmatrix} 
    I_{s} & \BZero_{s \times k} \\
    B & T \\
    \end{pmatrix}$ and the vectors $h$, where the tuple $(d_1, d_2, \dots, d_k)$ forms the diagonal of $T$;
        \item For any fixed $H$, enumerate the vectors $c \in \paral(-H^\top)$;
        \item For any fixed triplet $(H, h, c)$, enumerate all the values $c_0$ such that the corresponding simplex stays empty.
    \end{enumerate}

    The family of all tuples $(d_1, d_2, \dots, d_k)$ can be enumerated by the following way: assuming that values $d_1, d_2, \dots, d_j$ are already chosen from $\DS$ and $d_1 \cdot \ldots \cdot d_j < \Delta$, we chose $d_{j+1}$ such that the product $d_1 \cdot \ldots \cdot d_j \cdot d_{j+1}$ divides $\Delta$. Since this procedure consists of at most $\log_2(\Delta)$ steps, the enumeration complexity can be roughly bounded by
    \begin{multline*}
        \frac{\Delta}{2^0} \cdot \frac{\Delta}{2^1} \cdot \ldots \cdot \frac{\Delta}{2^{\lfloor \log_2(\Delta) \rfloor-1}} \leq \Delta^{\log_2(\Delta)} / 2^{\frac{(\log_2(\Delta)-2)(\log_2(\Delta)-1)}{2}} = \\
        = \frac{1}{2} \cdot \Delta^{3/2} \cdot \Delta^{\log_2(\Delta)/2}.
    \end{multline*}
    Here the inequalities $x - 1 \leq \lfloor x \rfloor \leq x$ has been used.
    
    For a chosen tuple $(d_1, d_2, \dots, d_k)$, the lower-triangular matrices $T \in \ZZ^{k \times k}$ can be enumerated by the following way. We put $(d_1, d_2, \dots, d_k)$ into the diagonal of $T$. When the diagonal of $T$ is defined, for any $j \in \intint k$ and $i \in \intint[j+1]{k}$, we need to choose $T_{i j} \in \intint[0]{T_{i i}-1}$. Note that, for any $j \in \intint k$, it holds $\prod_{i = j+1}^k T_{i j} \leq \Delta / 2^j$. Consequently, for a given diagonal, the matrices $T$ can be enumerated with $\Delta^{\log_2(\Delta)/2}$ operations. Together with the enumeration of diagonals, the arithmetic cost is $\Delta^{3/2} \cdot \Delta^{\log_2(\Delta)}$.

    To finish enumeration of the matrices $H$, we need to enumerate the matrices $B \in \ZZ^{k \times s}$. Let $\BS = \bigl\{x \in \ZZ^k \colon 0 \leq x < T_{i i},\,\text{for $i \in \intint k$} \bigr\}$. Then, since the columns of $B$ are lexicographically sorted, they are in one to one correspondence with the multi-subsets of $\BS$. Since $|\BS| = \Delta$, the total number of matrices $B$ is 
    $$
    \binom{s+\Delta-1}{\Delta-1} \leq \binom{n+\Delta-1}{\Delta-1}.
    $$ The enumeration of such objects is straightforward.
    
    Consequently, the arithmetic complexity to enumerate the matrices $H = \begin{pmatrix} 
    I_{s} & \BZero_{s \times k} \\
    B & T \\
    \end{pmatrix}$ is bounded by $\binom{n+\Delta-1}{\Delta-1} \cdot \Delta^{\log_2(\Delta) + 3/2}$. Since $h \in \BS$, the vectors $h$ can be enumerated just with $O(\Delta)$ operations. Therefore, the total number of operations to enumerate subsystems $H x \leq h$ is bounded by
    \begin{equation}\label{square_system_enum}
        \binom{n+\Delta-1}{\Delta-1} \cdot \Delta^{\log_2(\Delta)+5/2}.
    \end{equation}

    Due to \cite[Lemma~6]{FPT_Grib} (see also \cite[Lemma~9]{HalfOpen}, \cite{IntroEmptyLatticeSimplicies} or \cite[Lemma~1]{OnCanonicalProblems_Grib}), the vectors $c \in \paral(-H^{\top})$ can be enumerate with $O\bigl(n \cdot \Delta \cdot \min\{n,\log_2(\Delta)\}\bigr)$ operations.

    Now, let us discuss how to enumerate the values of $c_0$. Let $v = H^{-1} h$ and denote 
    $$
    \SC(c_0) = \Bigl\{x \in \RR^n \colon \binom{H}{c^\top} x \leq \binom{h}{c_0} \Bigr\}.
    $$ Note that if $\SC(c_0)$ is full-dimensional, then it forms a simplex, and $v$ is the opposite vertex of $\SC(c_0)$ with respect to the facet induced by the inequality $c^\top x \leq c_0$. Consider the following two cases: $v \notin \ZZ^n$ and $v \in \ZZ^n$.
    
    {\bf The case $v \notin \ZZ^n$.} Consider the problem
    \begin{gather}
        c^\top x \to \min\notag\\
        \begin{cases}
        H x \leq h \\
        x \in \ZZ^n,
        \end{cases}\label{group_prob1}
    \end{gather}
    and let $f^*$ be the optimal value of the objective function. Due to \cite[Theorem~12]{OnCanonicalProblems_Grib}, the problem \eqref{group_prob1} can be solved with $O\bigl(n+\min\{n, \Delta\} \cdot \Delta \cdot \log(\Delta)\bigr)$ operations. Denote $$
    l^* = \begin{cases}
        \lceil c^\top v \rceil,\quad\text{if $c^\top v \notin \ZZ$}\\
        c^\top v +1,\quad\text{if $c^\top v \notin \ZZ$}.
    \end{cases}
    $$
    Since $v \notin \ZZ^n$, and, due to Lemma \ref{form_lm}, we have $0 < f^* - l^* \leq \Delta$. By the construction, for any value of $c_0 \in \intint[l^*]{f^*-1}$, the set $\SC(c_0)$ forms an empty simplex. Due to the definition of $f^*$, for the values $c_0 \geq f^*$, the simplex $\SC(c_0)$ is not empty. Whenever, for $c_0 \leq l^* - 1$, we have $\SC(c_0) = \emptyset$. Therefore, for any $c_0 \in \intint[l^*]{f^*-1}$, we put the corresponding simplex $\SC(c_0)$ into the family $\SScr$.

    {\bf The case $v \in \ZZ^n$.} Consider the following problem
    \begin{gather}
        c^\top x \to \min\notag\\
        \begin{cases}
        H x \leq h \\
        x \in \ZZ^n\setminus\{\BZero\},
        \end{cases}\label{group_prob2}
    \end{gather}
    and let again $f^*$ be the optimal value of the objective function. To find $f^*$ we can solve the following set of problems: for any $j \in \intint n$, the problem is  
    \begin{gather}
        c^\top x \to \min\notag\\
        \begin{cases}
        H x \leq h - e_j \\
        x \in \ZZ^n.
        \end{cases}\label{group_prob3}
    \end{gather}
    Clearly, the value of $f^*$ can be identified as the minimum of optimum objectives of the defined set of problems. As it was already mentioned, each of these $n$ problems can be solved with $O\bigl(n + \min\{n, \Delta\} \cdot \Delta \cdot \log(\Delta)\bigr)$ operations. By the construction, for $c_0 < f^*$, the set $\SC(c_0)$ does not form a lattice-simplex or just empty. For $c_0 > f^*$, even if $\SC(c_0)$ forms a lattice-simplex, it is not empty. Consequently, only the simplex $\SC(f^*)$ is interesting. Next, if $\vertex\bigl(\SC(f^*)\bigr) \subseteq \ZZ^n$, we put the simplex $\SC(c_0)$ into the family $\LS$. In the opposite case, we just skip the current value of the vector $c$, and move to the next value (if it exists).  

    Due to Lemma \ref{norm_prop_lm}, for a fixed vector $c$, there are at most $\Delta$ possibilities to chose $c_0$. Consequently, with the proposed algorithms the complexity to enumerate the values of $c$ and $c_0$ is bounded by $\poly(n,\Delta)$. Therefore, recalling the complexity bound \eqref{square_system_enum} to enumerate the systems $H x \leq h$, the total enumeration complexity is the same as in the lemma definition, which completes the proof. 
\end{proof}

The following lemma helps to check the unimodular equivalence relation of two simplices.
\begin{lemma}\label{equiv_lm}
    Let $\SC$ be an arbitrary simplex defined by a system $A x \leq b$ with a $\Delta$-modular matrix $A$. Then, there exists an algorithm, which finds the set of all simplices of the family $\SScr$ (or $\LS$, respectively) that are unimodular equivalent to $\SC$. The algorithm complexity is
    $$
    n^{\log_2(\Delta)} \cdot \Delta! \cdot \poly(\phi),
    $$ where $\phi$ is encoding length of $A x \leq b$.
\end{lemma}
\begin{proof}
    Denote the set of resulting simplices by $\RS$. Due to the second proposition of Lemma \ref{form_lm}, we can use the following simple scheme to generate the family $\RS$:
    \begin{enumerate}
        \item Enumerate all the bases $\BC$ of $A$ such that $|\det(A_{\BC})| = \Delta$;
        

        \item For a fixed base $\BC$, use the normalization algorithm of Lemma \ref{form_lm} to the system $\binom{A_{\BC}}{A_{l}} x \leq \binom{b_{\BC}}{b_{l}}$, where $l$ is the row-index of the line of $A x \leq b$, which is not contained in $\BC$. Let
        $$
     \binom{H}{c^\top} x \leq \binom{h}{c_0},\quad\text{where $H = \begin{pmatrix} 
     I_{s} & \BZero_{s \times k} \\
     B & T \\
     \end{pmatrix}$.}
     $$ be the resulting normalized system;
        
        \item Enumerate all the $n \times n$ permutation matrices $P$;
        
        \item For a fixed $P$, use normalization algorithm of Lemma \ref{form_lm} to the system
        $$
        \binom{P H}{c^\top} x \leq \binom{P h}{c_0}.
        $$ The last step will produce a new normalized system that is unimodular equivalent to $\SC$, put it into $\RS$.
    \end{enumerate}
    This approach will generate the family $\RS$, but the enumeration of all permutations is expensive with respect to $n$. We will show next that it is not necessary to enumerate all the permutation matrices, and it is enough to enumerate only a relatively small part of them.

    Let us fix a base $\BC$ of $A$, which can be done in $n+1$ ways, and consider the normalized system $\binom{H}{c^\top} x \leq \binom{h}{c_0}$ that was produced by the second step of the scheme. Note that the rows of $H$ can be partitioned into two parts: each row of the first part is a row of $n \times n$ identity matrix $I_n$, each row of the second part is a row of the matrix $\bigl(B\;T\bigr)$. Let $\IC$ be the set of indices of the first part, and $\JC$ be the set of indices of the second. 
    Any permutation $\pi$ of the rows of $A_{\BC}$ can be generated as follows: 
    \begin{enumerate}
        \item Chose $|\JC|$ positions from $n$ positions where the rows with indexes $\JC$ will be located, which can be done in $\binom{n}{k}$ ways;

        \item Permute rows indexed by $\JC$, which can be done in $k!$ ways;

        \item Permute the rows with indexes in $\IC$.
    \end{enumerate}
    Let us take a closer look at permutations of the third type. We call a permutation $\pi$ \emph{redundant}, if $P_{\pi}  H = H Q$, for some unimodular matrix $Q \in \ZZ^{n \times n}$. Clearly, all the redundant permutations $\pi$ of the rows with indices in $\IC$ can be skipped, because $H$ and $P_{\pi} H$ have the same HNFs. 
    There is a natural one to one correspondence between the rows with indices in $\IC$ and columns of $B$: any row-permutation of this rows can be undone by a corresponding column-permutation of $B$. Since, for any $i \in \intint k$ and $j \in \intint s$, $B_{i j} \leq T_{i i}$ and $\prod_{i = 1}^k T_{i i} = \Delta$, the matrix $B$ has at most $\Delta$ different columns. It is easy to see now that any row-permutation $\pi$ of row indices that correspond to equal columns of $B$ are redundant. Therefore, it is sufficient to consider only the row-permutations of $\IC$ that correspond to different columns of $B$. Since $B$ consists of at most $\Delta$ different columns, there are at most $\Delta!$ of such permutations. It is relatively easy to compute these permutations: we just need to locate the classes of equal rows in $B$, locate any representative index of each class, and enumerate the permutations of these representatives by a standard way.


    Combining permutations of all three types and the complexity to construct the HNF, which is $\poly(\phi)$, the complexity of the procedure for the fixed base $\BC$ is
    $$
    \binom{n}{k} \cdot k! \cdot \Delta! \cdot \poly(\phi) \quad\lesssim\quad n^{\log_2(\Delta)} \cdot \Delta! \cdot \poly(\phi).
    $$
    Since there are at most $n+1$ ways to chose $\BC$, the total algorithm complexity is the same as in the lemma definition, which finishes the proof.
\end{proof}

\subsection{Finishing the proof of Theorems \ref{equiv_th} and \ref{enum_th}}\label{finishing_subs}

The following proposition is straightforward, but it is convenient to emphasise it for further use.
\begin{proposition}\label{equiv_prop} The simplices $\SC_1$ and $\SC_2$ are equivalent if and only if, then there exists a simplex $\SC'$, defined by a normalized system $A x \leq b$, such that $\SC_1$ and $\SC_2$ are both unimodular equivalent to $\SC'$.
\end{proposition}

{\bf The proof of Theorem \ref{equiv_th}.}
\begin{proof}
    First of all, we construct a unimodular equivalent normalized system for $\TC$. Due to Lemma \ref{form_lm}, it can be done by a polynomial-time algorithm.

    Next, we use Lemma \ref{equiv_lm} to construct the family $\RS$ of all the normalized systems that are equivalent to $\SC$. Due to Proposition \ref{equiv_prop}, if the simplices $\SC$ and $\TC$ are unimodular equivalent, then the family $\RS$ must contain the normalized system of $\TC$. To store the set $\RS$ effectively, we can use any well-balanced search-tree data structure with logarithmic-cost search, insert and delete operations. Clearly, we can compare two normalized systems just by representing them as two vectors of length $(n+1)^2$ and using the lexicographic order on these vectors. Due to complexity bound of Lemma \ref{equiv_lm}, and since $\log_2\bigl(|\RS|\bigr) = \poly(n, \Delta)$, the full algorithm complexity is the same as in the theorem definition, which completes the proof.
\end{proof}

{\bf The proof of Theorem \ref{enum_th}.}
\begin{proof}
    First of all, using Lemma \ref{enum_lm}, we generate the families $\SScr$ and $\LS$ of $\Delta$-modular empty simplicies and empty lattice-simplices respectively. As it was shown by Lemma \ref{enum_lm}, all the unimodular equivalence classes are already represented by $\SScr$ and $\LS$, but some different normalized systems in $\SScr$ or $\LS$ can represent the same equivalence class. Let us show how to remove such duplicates from $\SScr$, the same procedure works for $\LS$. 
    
    To store the set $\SScr$ effectively, we can use any well-balanced search-tree data structure with logarithmic-cost search, insert and delete operations. We can compare two normalized systems just by representing them as two vectors of length $(n+1)^2$ and using the lexicographic order on these vectors. Due to Lemma \ref{enum_lm}, we have $|\SScr| = O\bigl(\frac{n+\Delta}{\Delta}\bigr)^{\Delta-1} \cdot n^{O(1)}$. So, all operations with the search-tree can be performed with $\poly(n,\Delta)$ operations cost. 
    
    We enumerate all the normalized systems $\SC \in \SScr$, and using algorithm of Lemma \ref{equiv_lm}, for each $\SC$, enumerate all the normalized systems in $\SScr$ that are unimodular equivalent to $\SC$. For each system that was generated by this way, we remove the corresponding entry from the search-tree that represents $\SScr$. Due to Proposition \ref{equiv_prop}, after this procedure, all the duplicates will be removed from the search-tree, and only the normalized systems representing the unique equivalence classes will remain. Clearly, the total algorithm complexity equals to the product of the algorithm complexities of Lemmas \ref{enum_lm} and \ref{equiv_lm}:
    \begin{multline*}
        O\left(\frac{n+\Delta}{\Delta}\right)^{\Delta-1} \cdot n^{\log_2(\Delta)+O(1)} \cdot \Delta! = \\
        = O\bigl(n + \Delta\bigr)^{\Delta-1} \cdot n^{\log_2(\Delta)+O(1)},
    \end{multline*}
    which completes the proof.
\end{proof}

\section*{Conclusion}
The paper considers two problems:
\begin{itemize}
    \item The problem to enumerate all empty $\Delta$-modular simplices and empty $\Delta$-modular lattice-simplices modulo the unimodular equivalence relation;

    \item And the problem to check the unimodular equivalence of two $\Delta$-modular simplices that are not necessarily empty.
\end{itemize}
It was shown, assuming that the value of $\Delta$ is fixed, that the both problems can be solved by polynomial-time algorithms.

These results can be used to construct a data base containing all unimodular equivalence classes of empty $\Delta$-modular simplices and empty $\Delta$-modular lattice-simplices, for small values of $\Delta$ and moderate values of $n$. Due to \cite{FPT_Grib}, the lattice-width of empty simplex or empty lattice-simplex can be computed by an FPT-algorithm with respect to the parameter $\Delta$. Hence, the lattice-width is also can be precomputed for each simplex of the base. 

The author hopes that such a base and results of the paper will be helpful for studying the properties of empty simplices of both types, and general empty lattice-polytopes. Construction of the base and improvement of existing algorithms to deal with simplices is an interesting direction for further research.

\section*{Acknowledgments}
The results was prepared under financial support of Russian Science Foundation grant No 21-11-00194.

\bibliographystyle{splncs04}
\bibliography{grib_biblio}

\end{document}